\documentclass[a4paper]{article}

\pdfoutput=1
\usepackage[utf8]{inputenc}
\usepackage{lmodern}
\usepackage{array}
\usepackage{graphicx}
\usepackage{subfig}
\usepackage{amssymb, amsfonts, amsthm, amsmath}
\usepackage{enumerate}
\usepackage[english]{babel}
\usepackage[hidelinks]{hyperref}
\usepackage{tikz}
\usetikzlibrary{decorations}
\usetikzlibrary{decorations.pathreplacing}
\tikzset{individu/.style={draw,thick}}

\numberwithin{equation}{section}

\theoremstyle{plain}
\newtheorem{theorem}{Theorem}[section]
\newtheorem{corollary}[theorem]{Corollary}
\newtheorem{conjecture}[theorem]{Conjecture}
\newtheorem{lemma}[theorem]{Lemma}

\theoremstyle{definition}
\newtheorem{definition}[theorem]{Definition}

\theoremstyle{remark}
\newtheorem{remark}[theorem]{Remark}

\newcommand{\Z}{\mathbb{Z}}

\newcommand{\calR}{\mathcal{R}}
\newcommand{\calP}{\mathcal{P}}
\newcommand{\calQ}{\mathcal{Q}}
\newcommand{\calS}{\mathcal{S}}
\newcommand{\calC}{\mathcal{C}}
\newcommand{\calZ}{\mathcal{Z}}

\renewcommand{\bar}[1]{\overline{#1}}
\renewcommand{\tilde}[1]{\widetilde{#1}}
\renewcommand{\epsilon}{\varepsilon}
\renewcommand{\phi}{\varphi}

\newcommand{\Addresses}{{
  \bigskip
  \footnotesize

  \textsc{Unit\'e de Math\'ematiques Pures et Appliqu\'ees, \'Ecole normale sup\'erieure de Lyon, 46 all\'ee d'Italie, 69364 Lyon Cedex 07, France}\par\nopagebreak
  \textit{E-mail address}: \texttt{sanjay.ramassamy@ens-lyon.fr}

}}

\title{Extensions of partial cyclic orders, Euler numbers and multidimensional boustrophedons}
\author{Sanjay Ramassamy}
\date{\today}

\begin{document}

\maketitle

\begin{abstract}
We enumerate total cyclic orders on $\left\{1,\ldots,n\right\}$ where we prescribe the relative cyclic order of consecutive triples $(i,{i+1},{i+2})$, these integers being taken modulo $n$. In some cases, the problem reduces to the enumeration of descent classes of permutations, which is done via the boustrophedon construction. In other cases, we solve the question by introducing multidimensional versions of the boustrophedon. In particular we find new interpretations for the Euler up/down numbers and the Entringer numbers.
\end{abstract}

\section{Introduction}
\label{sec:introduction}

In this paper we enumerate some extensions of partial cyclic orders to total cyclic orders. In certain cases, this question is related to that of linear extensions of some posets.

A (linear) order on a set $X$ is a reflexive, antisymmetric and transitive binary relation on this set. When the set $X$ possesses a partial order, a classical problem is the search for (and the enumeration of) linear extensions, i.e. total orders on $X$ that are compatible with this partial order. Szpilrajn~\cite{S30} proved that such a linear extension always exists using the axiom of choice. It is possible to find a linear extension of a given finite poset in linear time (cf~\cite{CLR01} section 22.4). Brightwell and Winkler~\cite{BW91} proved that counting the number of linear extensions of a poset was $\#P$-complete.

Another type of order one can put on a set $X$ is a cyclic order, which is a ternary relation $Z \subset X^3$ satisfying the following conditions:
\begin{enumerate}
  \item $\forall x, y, z \in X$, $(x,y,z) \in Z \Rightarrow (y,z,x) \in Z$ (cyclicity).
  \item $\forall x ,y,z \in X$, $(x,y,z) \in Z \Rightarrow (z,y,x) \not \in Z$ (asymmetry).
  \item $\forall x,y,z,u \in X$, $(x,y,z) \in Z$ and $(x,z,u) \in Z \Rightarrow (x,y,u) \in Z$ (transitivity).
\end{enumerate}
A cyclic order may be partial or total (in the latter case, for any triple $(x,y,z)$ of distinct elements, either $(x,y,z)\in Z$ or $(z,y,x)\in Z$). The problem of studying the total cyclic orders extending a given partial cyclic order is much harder than its linear counterpart and has been subject to less investigations. Not every partial cyclic order admits an extension to a total cyclic order, as shown by Megiddo~\cite{M76}. Galil and Megiddo~\cite{GM77} proved that the problem of determining whether a given partial cyclic order admits an extension is NP-complete.

For any $n\geq1$, denote by $[n]$ the set $\left\{1,\ldots,n\right\}$. In this paper, we solve the question of the enumeration of total cyclic orders on $[n]$ which extend a partial cyclic order prescribing the relative order of any triple $(i,{i+1},{i+2})$. This is the cyclic counterpart of the classical linear extension problem considered in Subsection~\ref{subsec:descentclass}, where the relative order of every pair $(i,{i+1})$ is prescribed. We enumerate three types of total cyclic orders.

\begin{definition}
Fix $n\geq3$, $w=\epsilon_1\cdots\epsilon_{n-2}\in\left\{+,-\right\}^{n-2}$ and $\eta\in\left\{+,-\right\}$.
\begin{itemize}
 \item The set $\calP_w$ is the set of all total cyclic orders $Z$ on $[n]$ such that for any $1 \leq i \leq n-2$ verifying $\epsilon_i=+$ (resp. for any $1 \leq i \leq n-2$ verifying $\epsilon_i=-$), we have $(i,{i+1},{i+2})\in Z$ (resp. $({i+2},{i+1},i)\in Z$).
 \item The set $\calQ_w^+$ (resp. $\calQ_w^-$) is the set of all total cyclic orders $Z\in\calP_w$ such that $({n-1},n,1)\in Z$ (resp. $(1,n,{n-1})\in Z$).
 \item The set $\calR_w^{\eta,+}$ (resp. $\calR_w^{\eta,-}$) is the set of all total cyclic orders $Z\in\calQ_w^{\eta}$ such that $(n,1,2)\in Z$ (resp. $(2,1,n)\in Z$).
\end{itemize}
\end{definition}
See Figure~\ref{fig:cyclicorder4} for an example.

\begin{figure}[htpb]
\centering
\includegraphics[height=1.3in]{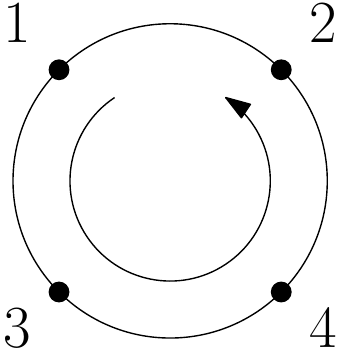}
\caption{A total cyclic order $Z$ on $[n]$ may be represented pictorially by placing all the elements of $[n]$ on a circle, in such a way that $(i,j,k)\in Z$ if and only if starting at $i$ and turning in the positive direction, one sees $j$ before $k$. Here $n=4$ and $Z$ is the set formed by the union of the four triples $(1,3,4),(3,4,2),(4,2,1)$ and $(2,1,4)$ with the eight possible cyclic permutations of these four triples. The arrow indicates the positive direction of rotation. The cyclic order $Z$ belongs to the following sets: $\calP_{-+}$, $\calQ_{-+}^+$ and $\calR_{-+}^{+-}$.}
\label{fig:cyclicorder4}
\end{figure}

Our main results concern the enumeration of total cyclic orders of each of the three above types. It is not hard to see that each such enumeration question is equivalent to the enumeration of total cyclic orders extending some given partial cyclic order on $[n]$. We show that these enumeration questions can always be solved by using linear recurrences which extend the classical boustrophedon construction used to compute the Entringer numbers. As a consequence, this provides an algorithm for computing the cardinalities of the sets $\calP_w$, $\calQ_w^{\eta}$ and $\calR_w^{\eta_1,\eta_2}$ which runs in polynomial time in the length of $w$, instead of the naive super-exponential algorithm which would consist in testing every total cyclic order to check if it belongs to one of these sets.

\subsection*{Outline of the paper}

In Section~\ref{sec:firstenumeration}, we state that the enumeration of each $\calP_w$ is equivalent to the enumeration of some descent class of permutations. As a consequence, we obtain new interpretations for the Euler and Entringer numbers in terms of cyclic orders. We prove these statements in Section~\ref{sec:proof1} by producing a specific bijection between total cyclic orders on $[n+1]$ and permutations of $[n]$. In Section~\ref{sec:secondenumeration}, we briefly recall the classical boustrophedon construction and we explain how to extend it to higher dimensions to enumerate the classes $\calQ_w^\eta$ and $\calR_w^{\eta_1,\eta_2}$. The proof of these linear recurrence relations can be found in Section~\ref{sec:proof2}. We finish in Section~\ref{sec:conjecture} by formulating a conjecture regarding the asymptotic densities of the classes $\calQ_w^\eta$ and $\calR_w^{\eta_1,\eta_2}$ inside the class $\calP_w$, when all the letters of $w$ are $+$, in the limit when the length of $w$ goes to infinity.

\section{Enumeration of \texorpdfstring{$\calP_w$}{Pw} and relation with the Euler and Entringer numbers}
\label{sec:firstenumeration}

The enumeration of the total cyclic orders in $\calP_w$ will be shown to be equivalent to the enumeration of descent classes of permutations, which we now introduce.

\subsection{Descent classes of permutations}
\label{subsec:descentclass}

For any $n\geq1$, denote by $\calS_n$ the set of permutations of $[n]$.
\begin{definition}
For any $n\geq2$, the \emph{descent pattern} of a permutation $\sigma\in\calS_n$ is the word $w=\epsilon_1\ldots\epsilon_{n-1}\in\left\{+,-\right\}^{n-1}$ such that for all $1\leq i\leq n-1$
\[
\epsilon_i=\begin{cases}
+ &\text{ if } \sigma(i+1)>\sigma(i), \\
- &\text{ if } \sigma(i+1)<\sigma(i).
\end{cases}
\]
The \emph{descent class} $\calS_w$ is defined to be the set of all permutations with descent pattern $w$.
\end{definition}
For example, the descent pattern of the permutation $\sigma$ whose one-line notation\footnote{Recall that the one-line notation for $\sigma\in\calS_n$ is $\sigma(1) \sigma(2) \cdots \sigma(n)$.} is $1 5 3 2 4$ is $+--+$.

To any word $w=\epsilon_1\ldots\epsilon_{n-1}\in\left\{+,-\right\}^{n-1}$, we associate the partial order $\prec_w$ on $[n]$ generated by the following relations: for any $1\leq i\leq n-1$, $i \prec_w {i+1}$ (resp. ${i+1} \prec_w i$) if $\epsilon_i=+$ (resp. $\epsilon_i=-$). Then the number of linear extensions of $\prec_w$ is $\#\calS_w$, hence the enumeration of descent classes is also a problem of enumeration of linear extensions of a poset.

The first formula for $\#\calS_w$ was obtained by MacMahon~\cite{M01}. For further formulas for $\#\calS_w$, see~\cite{V79} and the references therein.

A special descent class is the class of \emph{up/down permutations}, this is the case when the letters of $w$ in odd (resp. even) positions are all $+$ (resp. $-$). Andr\'e~\cite{A79} showed that if we write $F(x)=\sec x + \tan x$, then the number of up/down permutations in $\calS_n$ is the $n$-th Euler number $E_n:=F^{(n)}(0)$. One way to compute the Euler numbers is via the Entringer numbers $e_{n,i}$, which count the number of up/down permutations $\sigma\in\calS_n$ such that $\sigma(n)=i$. These Entringer numbers satisfy linear recurrence relations corresponding to the boustrophedon construction~\cite{E66} (see Subsection~\ref{subsec:boustrophedon} for more details).

\subsection{Connection with the enumeration of total cyclic orders}

For any $n\geq 1$ and $w\in\left\{+,-\right\}^n$, one can express the number of total cyclic orders in $\calP_w$ in terms of cardinalities of descent classes of permutations.
Define the involution $i$ on $\bigsqcup_{n\geq1} \left\{+,-\right\}^n$ which flips all the signs at even locations: $i(\epsilon_1\cdots\epsilon_n):=\epsilon'_1\cdots\epsilon'_n$ with
\[
\epsilon'_j=\begin{cases}
\epsilon_j &\text{ if } j \text{ is odd}, \\
-\epsilon_j &\text{ if } j \text{ is even}.
\end{cases}
\]
For example, $i(++--)=+--+$. Then the following holds:

\begin{theorem}
\label{thm:cyclicshape}
For integer $n\geq1$ and any word $w\in\left\{+,-\right\}^n$,
\begin{equation}
\#\calP_w=\#\calS_{i(w)}.
\end{equation}
\end{theorem}
The proof of Theorem~\ref{thm:cyclicshape} can be found in Section~\ref{sec:proof1}. For any permutation $\sigma\in\calS_{i(w)}$, the word $w$ is sometimes called the alternating descent pattern of $\sigma$, see for example~\cite{C08}.

\begin{remark}
We prove Theorem~\ref{thm:cyclicshape} by constructing a somewhat natural bijection between each $\calP_w$ and $\calS_{i(w)}$. It would have been easier to just show that the numbers $\#\calP_w$ verify the same linear recurrence relations as the ones verified by the numbers $\#\calS_{i(w)}$ (cf Subsection~\ref{subsec:boustrophedon}), but exhibiting a bijection between these sets provides a stronger connection between the classes $\calP_w$ and descent classes of permutations.
\end{remark}

As a corollary, taking $w$ to be a word with all the letters that are $+$, we deduce a new interpretation for the Euler numbers:

\begin{corollary}
\label{cor:Eulernumbers}
For any $n\geq1$, the Euler number $E_n$ is the number of total cyclic orders $Z$ on $[n+1]$ which verify $(i,{i+1},{i+2})\in Z$ for any $1\leq i \leq n-1$.
\end{corollary}

As a corollary of the proof of Theorem~\ref{thm:cyclicshape}, we also obtain a new interpretation for the Entringer numbers in terms of cyclic orders. Given a total cyclic order $Z$ on $[n]$, we define for any pair of distinct elements $(i,j)\in [n]$ the \emph{content} of the arc from $i$ to $j$ to be
\begin{equation}
c_Z(i,j):=\#\left\{x\in [n] | (i,x,j) \in Z \right\}.
\end{equation}
For example, $c_Z(3,4)=0$ for the total cyclic order $Z$ depicted in Figure~\ref{fig:cyclicorder4}. Then we have the following result:

\begin{corollary}
\label{cor:Entringernumbers}
For any $1 \leq i \leq n$, the Entringer number $e_{n,i}$ is the number of total cyclic orders $Z$ on $[n+1]$ verifying the following two conditions:
\begin{enumerate}
 \item for any $1\leq j \leq n-1$, we have $(j,{j+1},{j+2})\in Z$ ;
 \item the parameter $i$ satisfies
\[
i=
\begin{cases}
1+c_Z(n,{n+1}) &\text{ if } n \text{ is odd}, \\
1+c_Z({n+1},n) &\text{ if } n \text{ is even}.
\end{cases}
\]
\end{enumerate}
\end{corollary}

\section{Enumeration of \texorpdfstring{$\calQ_w^{\eta}$}{Q,w,eta} and \texorpdfstring{$\calR_w^{\eta_1,\eta_2}$}{R,w,eta1,eta2} and boustrophedons of higher dimensions}
\label{sec:secondenumeration}

\subsection{Boustrophedons}
\label{subsec:boustrophedon}

The classical way to compute the Euler and Entringer numbers is to set up a triangular array of numbers, called either the \emph{boustrophedon} or the Seidel-Entringer-Arnold triangle. Each line of the array is obtained by a linear transformation of the previous line, and the answer is read on the bottom line (see Figure~\ref{fig:boustrophedon}).

\begin{figure}[htpb]
\[
\begin{matrix}
&&&&1&&&& \\
&&&&&&&& \\
&&&0&&1&&& \\
&&&&&&&& \\
&&1&&1&&0&& \\
&&&&&&&& \\
&0&&1&&2&&2& \\
&&&&&&&& \\
5&&5&&4&&2&&0 \\
\end{matrix}
\]
\caption{Computation of the Entringer numbers $e_{n,i}$ for $1\leq i \leq n\leq5$ using the boustrophedon method. The number $e_{n,i}$ is the $i$-th number (counting from the left) on the $n$-th line. Here each entry of a line of even (resp. odd) index is obtained by taking the sum of the entries of the previous line that lie to its left (resp. right).}
\label{fig:boustrophedon}
\end{figure}

Viennot~\cite{V79} extended this construction to obtain the cardinality of any descent class $\calS_w$. Thus, by Theorem~\ref{thm:cyclicshape}, one can compute the cardinality of $\calP_w$ by means of a linear inductive process.

Unlike the case of $\calP_w$, it seems that we cannot reduce the question of enumerating $\calQ_w^{\eta}$ and $\calR_w^{\eta_1,\eta_2}$ to the enumeration of some nice or known class of permutations. For example, while $(\#\calP_{+^n})_{n\geq1}$ (where $+^n$ denotes the word of length $n$ with all letters equal to $+$) is the sequence of Euler numbers by Corollary~\ref{cor:Eulernumbers}, the sequences $(\#\calQ_{+^n}^+)_{n\geq1}$ and $(\#\calR_{+^n}^{+,+})_{n\geq1}$ are currently absent from the On-Line Encyclopedia of Integer Sequences~\cite{OEIS} (see Table~\ref{tab:firstterms} for the first few values). 

\begin{table}[htbp]
\centering
\begin{tabular}{|c|c|c|c|c|c|c|c|c|c|c|}
  \hline
  $n$ & 1& 2 & 3 & 4 & 5 & 6 & 7 & 8 & 9 & 10 \\
  \hline
  $\#\calP_{+^n}$ & 1 & 2 & 5 & 16 & 61 & 272 & 1385 & 7936 & 50521 & 353792 \\
  \hline
  $\#\calQ_{+^n}^+$ & 1 & 1 & 3 & 11 & 38 & 169 & 899 & 5047 & 31914 & 226205 \\
  \hline
  $\#\calR_{+^n}^{+,+}$ & 1 & 1 & 2 & 9 & 31 & 128 & 708 & 4015 & 24865 & 177444 \\
  \hline
\end{tabular}
\caption{The first ten values of the cardinalities of the sets $\calP_{+^n}$, $\calQ_{+^n}^+$ and $\calR_{+^n}^{+,+}$. The first sequence corresponds to the Euler up/down numbers. We formulate in Section~\ref{sec:conjecture} a conjecture regarding the asymptotic ratio of the terms of these sequences.}
\label{tab:firstterms}
\end{table}

However, we solve these enumeration questions by introducing linear recurrences that are higher-dimensional versions of the boustrophedon. The boustrophedon can be seen as the time-evolution of a sequence of numbers, where at time $t$ the sequence has length $t$ and the sequence at time $t+1$ is obtained from the sequence at time $t$ by a linear transformation.

The enumeration of $\calQ_w^{\eta}$ (resp. $\calR_w^{\eta_1,\eta_2}$) is done via the time-evolution of triangles of numbers (resp. tetrahedra of numbers), where at time $t$ the triangles (resp. tetrahedra) have side-length $t$ and the triangles (resp. tetrahedra) at time $t+1$ are obtained from the triangles (resp. tetrahedra) at time $t$ by a linear transformation. We will introduce a family of operators $\Phi_{a,b,c}$ and show that the recursions for $\calP_w$, $\calQ_w^{\eta}$ and $\calR_w^{\eta_1,\eta_2}$ can all be expressed simply using these operators.

Foata and Han~\cite{FH14} also studied the evolution of triangular arrays of numbers in order to enumerate Bi-Entringer numbers, which count up/down permutations $\sigma\in\calS_n$ with prescribed values of $\sigma(1)$ and $\sigma(n)$. Our $\Phi$ operators can also be used to describe their recurrence formulas.

\subsection{Enumeration of \texorpdfstring{$\calQ_w^{\eta}$}{Q,w,eta}}

In order to enumerate $\calQ_w^{\eta}$ or $\calR_w^{\eta_1,\eta_2}$, we need to refine these classes by fixing finitely many reference points and introducing a parameter to specify the content of each arc between two consecutive reference points. Given a total cyclic order $Z$ on a set $X$ and $(y_1,\ldots,y_p)$ a $p$-tuple of distinct elements in $X$, we define the \emph{multi-content} $\tilde{c}_Z(y_1,\ldots,y_p)$ to be the $p$-tuple
\begin{equation}
\tilde{c}_Z(y_1,\ldots,y_p):=(c_Z(y_1,y_2),c_Z(y_2,y_3),\ldots,c_Z(y_{p-1},y_p),c_Z(y_p,y_1)).
\end{equation}

For any $n\geq3, w\in\left\{+,-\right\}^{n-2}$ and nonnegative integers $i,j,k$ such that $i+j+k=n-3$, set
\begin{align}
f_{w,i,j,k}^+ &:= \# \left\{ Z\in\calQ_w^+ | \tilde{c}_Z({n-1},n,1)=(i,j,k) \right\} \\
f_{w,i,j,k}^- &:= \# \left\{ Z\in\calQ_w^- | \tilde{c}_Z(n,{n-1},1)=(i,j,k) \right\}.
\end{align}
These numbers provide a refined enumeration of $\calQ_w^+$ and $\calQ_w^-$ according to the content of each arc between the elements $n-1$, $n$ and $1$, playing the same role as the role played by the Entringer numbers for the Euler numbers. Just like the Entringer number, these numbers $f_{w,i,j,k}^+$ and $f_{w,i,j,k}^-$ satisfy some linear recurrence relations.

If $w\in\left\{+,-\right\}^n$, we denote by $w+$ (resp. $w-$) the word on $n+1$ letters obtained by adding the letter $+$ (resp $-$) at the end of the word $w$. We then have the following recurrence relations:

\begin{theorem}
\label{thm:extendedcoefficients}
For any $n\geq1$, $w\in\left\{+,-\right\}^n$ and nonnegative integers $i,j,k$ such that $i+j+k=n$, we have
\begin{align}
f_{w+,i,j,k}^+ &=\sum_{j'=0}^{j-1} f_{w,i+j-1-j',j',k}^- + \sum_{k'=0}^{k-1} f_{w,k-1-k',i+j,k'}^+ \label{eq:firstlinearrecurrence} \\
f_{w+,i,j,k}^- &=\sum_{i'=0}^{i-1} f_{w,i',j,i+k-1-i'}^+ \label{eq:secondlinearrecurrence} \\
f_{w-,i,j,k}^+ &=\sum_{i'=0}^{i-1} f_{w,i',i+j-1-i',k}^- \\
f_{w-,i,j,k}^- &=\sum_{k'=0}^{k-1} f_{w,i+k-1-k',j,k'}^+ + \sum_{j'=0}^{j-1} f_{w,j-1-j',j',i+k}^-. 
\end{align}
\end{theorem}

See Section~\ref{sec:proof2} for the proof of Theorem~\ref{thm:extendedcoefficients}. For fixed $n\geq1$, $w\in\left\{+,-\right\}^n$ and $\eta\in\left\{+,-\right\}$, the collection
\[
T_w^{\eta}:=\left\{f_{w,i,j,k}^{\eta}|i+j+k=n-1\right\}
\]
forms a triangular array of numbers, and Theorem~\ref{thm:extendedcoefficients} gives a recursive way to compute the cardinality of any $\calQ_w^{\eta}$: compute the evolution of the pair of triangular arrays $(T_{w'}^+,T_{w'}^-)$, for $w'$ a prefix of $w$ of increasing length, until one gets to $T_w^{\eta}$, then take the sum of all the entries of $T_w^{\eta}$ (see Figure~\ref{fig:triangleevolution}). The recurrences are initialized as follows:

\begin{equation}
\begin{array}{ccccc}
f_{+,0,0,0}^+ & = & f_{-,0,0,0}^- &=& 1 \\
f_{-,0,0,0}^+ & = & f_{+,0,0,0}^- &=& 0. \\
\end{array}
\end{equation}

\begin{figure}[htpb]
\[
\begin{array}{rccrc}
T_+^+=& 1 && T_+^-= &0 \\
&&& \\
&&& \\
&&& \\
T_{++}^+=
&\begin{matrix}
 & 0 & \\
 &&\\
0 & & 1\\
\end{matrix}
&&T_{++}^-=
&\begin{matrix}
 & 1 & \\
 &&\\
0 & & 0\\
\end{matrix} \\
&&& \\
&&& \\
&&& \\
T_{+++}^+=
&\begin{matrix}
&&0&&\\
&&&&\\
&1&&0&\\
&&&&\\
1&&0&&1\\
\end{matrix}
&&T_{+++}^-=
&\begin{matrix}
&&1&&\\
&&&&\\
&0&&1&\\
&&&&\\
0&&0&&0\\
\end{matrix} \\
&&& \\
&&& \\
&&& \\
T_{++++}^+=
&\begin{matrix}
&&&0&&&\\
&&&&&&\\
&&1&&1&&\\
&&&&&&\\
&1&&2&&1&\\
&&&&&&\\
1&&2&&1&&1\\
\end{matrix}
&&T_{++++}^-=
&\begin{matrix}
&&&1&&&\\
&&&&&&\\
&&1&&1&&\\
&&&&&&\\
&1&&0&&1&\\
&&&&&&\\
0&&0&&0&&0\\
\end{matrix}
\end{array}
\]
\caption{Pairs of triangular arrays of numbers of growing size used to enumerate $\calQ_{++++}^\eta$. The bottom, right and left sides of each triangle respectively correspond to $i=0$, $j=0$ and $k=0$. Taking the sum of the entries in $T_{++++}^+$, one obtains $\#\calQ_{++++}^+=11$, as indicated in Table~\ref{tab:firstterms}.}
\label{fig:triangleevolution}
\end{figure}

The linear recursions can be rewritten in a more compact way, by introducing the multivariate generating series of the numbers $f_{w,i,j,k}^+$ and $f_{w,i,j,k}^-$ and defining some linear operators acting on these generating functions. Fix $m\geq 2$ and $1\leq a,b,c \leq m$ to be three integers such that $b\neq c$ ($a$ may be equal to $b$ or $c$). We define the linear endomorphism $\Phi_{a,b,c}$ of $\Z[X_1,\ldots,X_m]$ by its action on monomials: for any $i_1,\ldots,i_m\geq0$,
\begin{equation}
\Phi_{a,b,c}\left(\prod_{\ell=1}^m X_\ell^{i_\ell} \right) = \left(\prod_{\substack{\ell=1 \\ \ell\notin\left\{b,c\right\}}}^m X_{\ell}^{i_\ell} \right)X_a^{i_b+1} \sum_{k=0}^{i_c} X_b^{i_c-k}X_c^{k}.
\end{equation}
Note that $\Phi_{a,b,c}$ maps any homogeneous polynomial of degree $d$ to a homogeneous polynomial of degree $d+1$.

For any $n\geq1$, $w\in\left\{+,-\right\}^n$ and $\eta\in\left\{+,-\right\}$, we form the generating function
\begin{equation}
Q_w^{\eta}(X_1,X_2,X_3):=\sum_{\substack{i,j,k\geq0 \\ i+j+k=n-1}} f_{w,i,j,k}^\eta X_1^i X_2^j X_3^k.
\end{equation}

Then Theorem~\ref{thm:extendedcoefficients} can be rewritten as follows.

\begin{theorem}
\label{thm:extendedgf}
For any $n\geq1$ and $w\in\left\{+,-\right\}^n$, we have
\begin{align}
Q_{w+}^+ &=\Phi_{2,2,1}(Q_w^-) + \Phi_{3,1,2}(Q_w^+) \\
Q_{w+}^- &=\Phi_{1,1,3}(Q_w^+) \\
Q_{w-}^+ &=\Phi_{1,1,2}(Q_w^-) \\
Q_{w-}^- &=\Phi_{3,3,1}(Q_w^+) + \Phi_{2,1,3}(Q_w^-).
\end{align}
\end{theorem}

\begin{remark}
\label{rem:classicalcase}
The $\Phi$ operators can be used to enumerate $\calP_w$, which by Theorem~\ref{thm:cyclicshape} corresponds to the classical boustrophedon and its extension by Viennot in~\cite{V79}. For any $n\geq3$, $w\in\left\{+,-\right\}^{n-2}$ and nonnegative integers $i,j$ such that $i+j=n-2$, set
\begin{equation}
e_{w,i,j}:=
\begin{cases}
\# \left\{Z\in\calP_w | \tilde{c}_Z({n-1},n)=(i,j)\right\} &\text{ if } n \text{ is even} \\
\# \left\{Z\in\calP_w | \tilde{c}_Z({n-1},n)=(j,i)\right\} &\text{ if } n \text{ is odd} 
\end{cases}
\end{equation}
and define the generating function
\begin{equation}
P_w(X_1,X_2):=\sum_{\substack{i,j\geq 0 \\ i+j=n-2}} e_{w,i,j} X_1^iX_2^j.
\end{equation}
Then we have
\begin{align}
P_{w+}&=
\begin{cases}
\Phi_{1,1,2}(P_w) &\text{ if } n \text{ is even} \\
\Phi_{2,2,1}(P_w) &\text{ if } n \text{ is odd} 
\end{cases} \\
P_{w-}&=
\begin{cases}
\Phi_{2,2,1}(P_w) &\text{ if } n \text{ is even} \\
\Phi_{1,1,2}(P_w) &\text{ if } n \text{ is odd} .
\end{cases}
\end{align}
\end{remark}

\begin{remark}
\label{rem:foatahan}
Foata and Han~\cite{FH14} introduced and studied Seidel triangle sequences, which are sequences $(A_n)_{n\geq1}$ of triangular arrays of numbers of growing size satisfying some linear recursion relations. One may reformulate their definition using the $\Phi$ operators as follows. Fix $H$ to be an arbitrary infinite triangular array of numbers (the $n$-th line of $H$ contains $n$ numbers). For any $n\geq1$, denote by $T_n(H)$ the triangular array of numbers of side-length $n$ defined as follows: the entries of $T_n(H)$ are constant along rows of constant $j$-coordinate and the last line of $T_n(H)$ is equal to the $n$-th line of $H$ (see Figure~\ref{fig:examplefh} for an example).

\begin{figure}[htpb]
\[
\begin{matrix}
&&h^{(3)}_3&& \\
&&&& \\
&h^{(3)}_2&&h^{(3)}_3& \\
&&&& \\
h^{(3)}_1&&h^{(3)}_2&&h^{(3)}_3 \\
\end{matrix}
\]
\caption{Triangle $T_3(H)$ when the third line of $H$ is given from left to right by $h^{(3)}_1,h^{(3)}_2,h^{(3)}_3$.}
\label{fig:examplefh}
\end{figure}

Then the Seidel triangle sequence built from $H$ is defined to be the sequence of triangular arrays $(A_n)_{n\geq1}$ such that $A_1=T_1(H)$ and for any $n\geq2$, $A_n=\Phi_{1,1,3}(A_{n-1})+T_n(H)$. Note that here we slightly abused the notation by applying $\Phi_{1,1,3}$ to a triangular array of numbers instead of its generating function.
\end{remark}

\subsection{Enumeration of \texorpdfstring{$\calR_w^{\eta_1,\eta_2}$}{R,w,eta1,eta2}}

In order to obtain linear recurrence relations to compute the cardinality of $\calR_w^{\eta_1,\eta_2}$, we need to define six types of subsets, distinguished according to the cyclic order of the four elements $1,2,n-1$ and $n$. If $Z$ is a total cyclic order on a finite set $X$ and $y_1,\ldots,y_p$ are $p$ distinct elements of $X$ (with $p\geq 3$), we say that $(y_1,\ldots,y_p)$ forms a $Z$-chain if for any $3 \leq i \leq p$, $(y_1,y_{i-1},y_i)\in Z$. We denote by $\calC_Z$ the set of $Z$-chains. Using this notation, we define for any $n\geq4$ and $w\in\left\{+,-\right\}^{n-2}$ the following sets:
\begin{align}
\calR_w^{(1)}&:=\left\{ Z\in\calP_w | (1,2,{n-1},n) \in \calC_Z  \right\} \\
\calR_w^{(2)}&:=\left\{ Z\in\calP_w | (1,{n-1},2,n) \in\calC_ Z  \right\} \\
\calR_w^{(3)}&:=\left\{ Z\in\calP_w | (1,{n-1},n,2) \in \calC_Z  \right\} \\
\calR_w^{(4)}&:=\left\{ Z\in\calP_w | (1,2,n,{n-1}) \in \calC_Z  \right\} \\
\calR_w^{(5)}&:=\left\{ Z\in\calP_w | (1,n,2,{n-1}) \in \calC_Z  \right\} \\
\calR_w^{(6)}&:=\left\{ Z\in\calP_w | (1,n,{n-1},2) \in \calC_Z  \right\}.
\end{align}
Note that for any $n\geq2$ and $w\in\left\{+,-\right\}^n$,
\begin{align}
\calR_w^{+,+}&=\calR_w^{(1)} \sqcup \calR_w^{(2)} \\
\calR_w^{+,-}&=\calR_w^{(3)} \\
\calR_w^{-,+}&=\calR_w^{(4)} \\
\calR_w^{-,-}&=\calR_w^{(5)} \sqcup \calR_w^{(6)}.
\end{align}

While the enumeration of $\calQ_w^+$ and $\calQ_w^-$ was performed by refining according to the multi-content associated with the reference points $1$, $n-1$ and $n$, we enumerate each class $\calR_w^{(6)}$ with $1\leq \alpha \leq 6$ by refining according to the multi-content associated with the four reference points $1$, $2$, $n-1$ and $n$. For any $n\geq4$ and nonnegative integers $i,j,k,\ell$ such that $i+j+k+\ell=n-4$, set:

\begin{align}
g_{w,i,j,k,\ell}^{(1)} &:=
\# \left\{ Z\in\calR_w^{(1)} | \tilde{c}_Z({n-1},n,1,2)=(i,j,k,\ell) \right\} \\
g_{w,i,j,k,\ell}^{(2)} &:=
\# \left\{ Z\in\calR_w^{(2)} | \tilde{c}_Z({n-1},2,n,1)=(i,\ell,j,k) \right\} \\
g_{w,i,j,k,\ell}^{(3)} &:= 
\# \left\{ Z\in\calR_w^{(3)} | \tilde{c}_Z(n,2,1,{n-1})=(i,j,k,\ell) \right\} \\
g_{w,i,j,k,\ell}^{(4)} &:= 
\# \left\{ Z\in\calR_w^{(4)} | \tilde{c}_Z(n,{n-1},1,2)=(i,j,k,\ell) \right\} \\
g_{w,i,j,k,\ell}^{(5)} &:=
\# \left\{ Z\in\calR_w^{(5)} | \tilde{c}_Z(n,2,{n-1},1)=(i,\ell,j,k) \right\} \\
g_{w,i,j,k,\ell}^{(6)} &:=
\# \left\{ Z\in\calR_w^{(6)} | \tilde{c}_Z({n-1},2,1,n)=(i,j,k,\ell) \right\}.
\end{align}

For any $1\leq \alpha \leq 6$, $n\geq2$ and $w\in\left\{+,-\right\}^n$, the collection
\[
\left\{g_{w,i,j,k,\ell}^{(\alpha)} | i+j+k+\ell = n-2\right\}
\]
forms a tetrahedral array of numbers. We provide linear recurrence formulas for these arrays directly in the language of generating functions (we skip the less compact formulation in terms of sequences indexed by $i,j,k,\ell$). For any $1 \leq \alpha \leq 6$, $n\geq2$ and $w\in\left\{+,-\right\}^n$, we set
\begin{equation}
R_w^{(\alpha)}(X_1,X_2,X_3,X_4):=\sum_{\substack{i,j,k,\ell\geq0 \\ i+j+k+\ell=n-2}} g_{w,i,j,k,\ell}^{(\alpha)} X_1^i X_2^j X_3^k X_4^\ell.
\end{equation}

Then we can express the recurrence relations among the $R_w^{(\alpha)}$'s by using the operators $\Phi_{a,b,c}$ defined above:

\begin{theorem}
\label{thm:fullgf}
For any $n\geq2$ and $w\in\left\{+,-\right\}^n$, we have
\begin{align}
R_{w+}^{(1)}&=\Phi_{4,1,2}(R_{w}^{(1)} ) +  \Phi_{3,1,2}(R_{w}^{(2)} ) + \Phi_{2,2,1}(R_{w}^{(4)} ) \\
R_{w+}^{(2)}&=\Phi_{3,4,2}(R_{w}^{(3)} ) +  \Phi_{2,2,4}(R_{w}^{(5)} ) \\
R_{w+}^{(3)}&=\Phi_{3,4,1}(R_{w}^{(3)} ) +  \Phi_{2,4,1}(R_{w}^{(5)} ) + \Phi_{1,1,4}(R_{w}^{(6)} ) \\
R_{w+}^{(4)}&=\Phi_{1,1,4}(R_{w}^{(1)} )   \\
R_{w+}^{(5)}&=\Phi_{4,1,3}(R_{w}^{(1)} )  + \Phi_{1,1,3}(R_{w}^{(2)} )  \\
R_{w+}^{(6)}&=\Phi_{4,4,3}(R_{w}^{(3)} )  \\
R_{w-}^{(1)}&=\Phi_{1,1,2}(R_{w}^{(4)} )  \\
R_{w-}^{(2)}&=\Phi_{1,4,2}(R_{w}^{(6)} ) + \Phi_{4,4,2}(R_{w}^{(5)} ) \\
R_{w-}^{(3)}&=\Phi_{4,4,1}(R_{w}^{(6)} )  \\
R_{w-}^{(4)}&=\Phi_{2,1,4}(R_{w}^{(4)} ) + \Phi_{3,1,4}(R_{w}^{(2)} ) + \Phi_{4,4,1}(R_{w}^{(1)} ) \\
R_{w-}^{(5)}&= \Phi_{2,1,3}(R_{w}^{(4)} ) + \Phi_{3,3,1}(R_{w}^{(2)} ) \\
R_{w-}^{(6)}&=\Phi_{1,4,3}(R_{w}^{(6)} ) + \Phi_{2,4,3}(R_{w}^{(5)} ) + \Phi_{3,3,4}(R_{w}^{(3)} ).
\end{align}
\end{theorem}

The recurrences are initialized as follows:
\[
R_{++}^{(1)}=R_{--}^{(2)}=R_{-+}^{(3)}=R_{+-}^{(4)}=R_{++}^{(5)}=R_{--}^{(6)}=1
\]
and all other $R_{\epsilon_1,\epsilon_2}^{(\alpha)}$ are set to $0$. See Section~\ref{sec:proof2} for the proof of Theorem~\ref{thm:fullgf}. 

\section{Proof of Theorem~\ref{thm:cyclicshape}}
\label{sec:proof1}

\subsection{Proof idea}

A total cyclic order $Z$ on $[n]$ can be viewed as a way to place the numbers $1$ to $n$ on a circle, such as on Figure~\ref{fig:cyclicorder4}, where only the relative positions of the numbers matter and are prescribed by $Z$. A permutation in $\calS_n$ can be viewed as a diagram of dots where only the relative positions of the dots matter, such as on Figure~\ref{fig:permutation4}.

\begin{figure}[htpb]
\centering
\includegraphics[height=.8in]{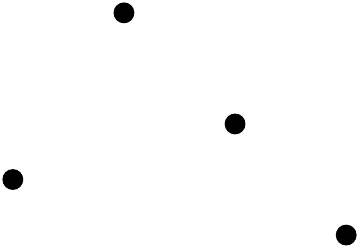}
\caption{This dot diagram represents the permutation $2431$, because when reading the dots from left to right, the first dot is the second lowest, the second is the highest, the third is the second highest and the fourth is the lowest.}
\label{fig:permutation4}
\end{figure}

To prove Theorem~\ref{thm:cyclicshape} we will construct for every $[n]$ a bijection from total cyclic orders on $[n+1]$ to permutations on $[n]$ then show that this bijection maps each $\calP_w$ to $\calS_{i(w)}$. This bijection is constructed by induction on $n$. We grow the total cyclic order by adding the numbers $1,2,\ldots,n+1$ one after the other. Simultaneously we grow the dot diagram, by adding the $i$-th dot at the same time as we add the number $i+1$ to the cyclic order. At the beginning of the process, the cycle only contains the numbers $1$ and $2$ and the dot diagram only contains a single dot. Assume we already have $j$ elements in the cycle with $j\geq2$ and $j-1$ dots in the diagram. We divide the space in which the dot diagram lives into $j$ regions separated by horizontal boundaries, with one horizontal boundary at the height of each currently present dot. The slices are numbered from $0$ to $j-1$, either from bottom to top if $j$ is odd or from top to bottom if $j$ is even. See Figure~\ref{fig:slices} for an example.

\begin{figure}[htbp]
\centering
\subfloat[Slicing the plane into four numbered regions when $j=4$.]{\label{fig:slices3}\includegraphics[height=1.5in]{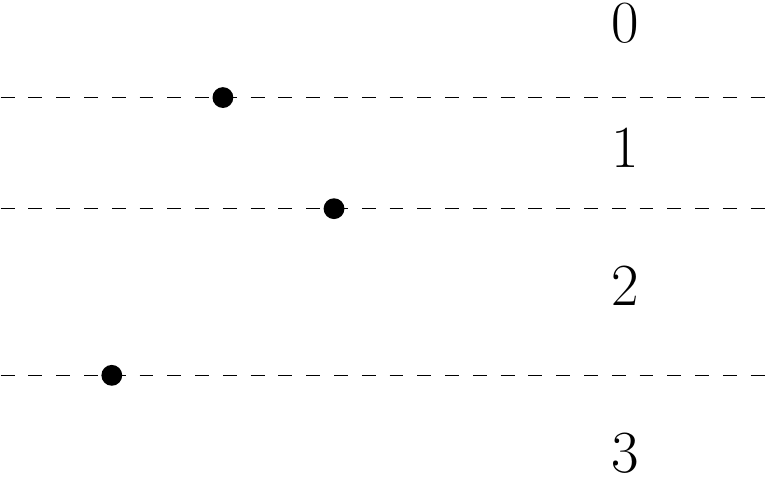}}
\hspace{\stretch{1}}
\subfloat[Slicing the plane into five numbered regions when $j=5$.]{\label{fig:slices4}\includegraphics[height=1.5in]{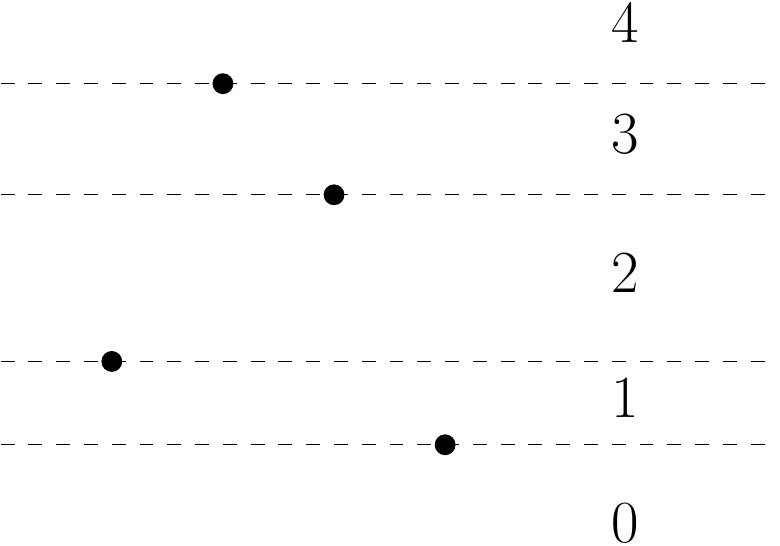}}
\caption{Slicing the plane and numbering the corresponding regions when $j$ is even or odd.}
\label{fig:slices}
\end{figure}

Let $k$ denote the number of elements which are on the arc from $j$ to $j+1$ at the time when the cycle contains $j+1$ elements. We add the $j$-th dot in the region number $k$ of the dot diagram, to the right of all the dots already present. See Figure~\ref{fig:dotaddition} for an example.

\begin{figure}[htbp]
\centering
\subfloat[Adding the number $5$ to a cycle already containing $4$ elements.]{\label{fig:adding5}\includegraphics[height=2in]{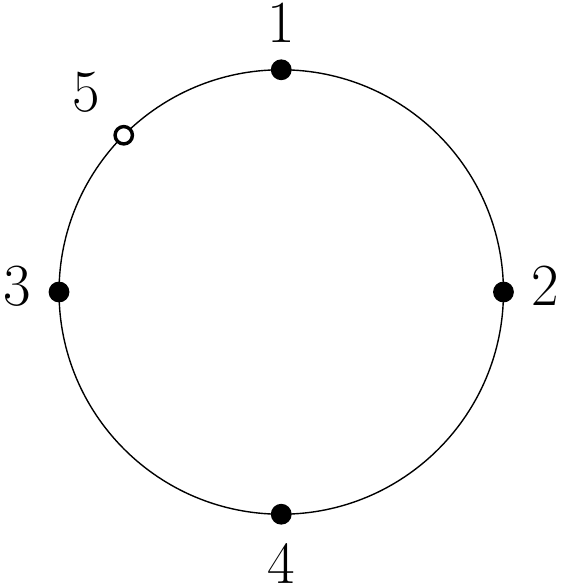}}
\hspace{\stretch{1}}
\subfloat[The dot diagram corresponds to the permutation $321$ before the addition of the fourth dot (represented here by the hollow dot). After adding the fourth dot, it corresponds to the permutation $4312$.]{\label{fig:addingdot}\includegraphics[height=1.2in]{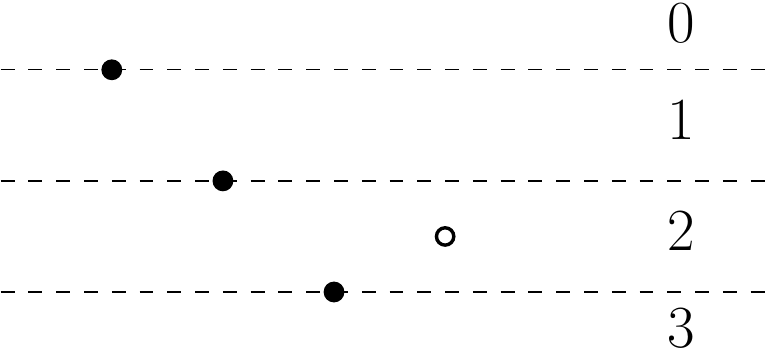}}
\caption{In this example, $j=4$ and $k=2$, because the numbers $2$ and $1$ are on the arc between $4$ and the newly added number $5$. Thus the newly added dot in the dot diagram is in the region number $2$.}
\label{fig:dotaddition}
\end{figure}

In the remainder of this section, we first define in Subsection~\ref{subsec:cyclicdescentclass} a one-to-one correspondence between each $\calP_w$ and a certain class of cyclic permutations on $[n]$, which we call the cyclic descent class $\calC_w$, then we formalize in Subsection~\ref{subsec:bijection} the map described above and finally in Subsections~\ref{subsec:lemma1} and~\ref{subsec:lemma2} we show that this map is a bijection from each $\calP_w$ and $\calS_{i(w)}$.

\subsection{Cyclic descent classes}
\label{subsec:cyclicdescentclass}

There is a one-to-one correspondence between permutations $\sigma\in\calS_n$ and total orders $<_{\sigma}$ on $[n]$, by setting $i <_{\sigma} j$ if and only if $\sigma(i)<\sigma(j)$. As observed in Subsection~\ref{subsec:descentclass}, this reduces the problem of enumerating certain linear extensions of posets on $[n]$ to enumerating descent class of permutations in $\calS_n$.

By analogy with the linear setting, we define a one-to-one correspondence between cyclic permutations $\pi\in\calC_n$ and total cyclic orders $Z$ on $[n]$. For any $n\geq3$, let $\calZ_n$ denote the set of all total cyclic orders on $[n]$. Define the map
\[
\zeta:\bigsqcup_{n\geq3} \calC_n \rightarrow \bigsqcup_{n\geq3} \calZ_n
\]
where for any $\pi\in\calC_n$, $\zeta(\pi)$ is defined as follows: $({i_1},{i_2},{i_3})\in\zeta(\pi)$ if and only if $i_2=\pi^k(i_1)$ and $i_3=\pi^{\ell}(i_1)$ for some $1 \leq  k< \ell \leq n-1$. In words, for the total cyclic order $\zeta(\pi)$, ${\pi(i)}$ is the next element after $i$ when turning in the positive direction. For example, if $n=4$ and $\pi=[1,3,4,2]$, then $\zeta(\pi)$ is the total cyclic order depicted in Figure~\ref{fig:cyclicorder4}. The map $\zeta$ is clearly a bijection from $\calC_n$ to $\calZ_n$ for any $n\geq3$.

Continuing the analogy with the linear setting, we define cyclic descent classes.
\begin{definition}
Fix $n\geq3$. The \emph{cyclic descent pattern of a cyclic permutation} $\pi\in\calC_n$ is the word $w=\epsilon_1\ldots\epsilon_{n-2}\in\left\{+,-\right\}^{n-2}$ such that for all $1\leq i\leq n-2$
\[
\epsilon_i=\begin{cases}
+ \text{ if } (i,{i+1},{i+2}) \in \zeta(\pi), \\
- \text{ if } ({i+2},{i+1},i) \in \zeta(\pi).
\end{cases}
\]
The \emph{cyclic descent class} $\calC_w$ is defined to be the set of all cyclic permutations with cyclic descent pattern $w$.
\end{definition}

Clearly $\zeta$ maps $\calC_w$ to $\calP_w$. Thus, to prove Theorem~\ref{thm:cyclicshape}, it suffices to show that for any $n\geq1$ and $w\in\left\{+,-\right\}^n$, we have $\# \calC_w = \#\calS_{i(w)}$.

\subsection{Construction of a bijection}
\label{subsec:bijection}

Before describing the construction of $F$, we need some preliminary remarks about notations for permutations. For any $n\geq3$, we define the map $\partial:\calC_n\rightarrow\calC_{n-1}$ as follows: for any $\pi\in\calC_n$, $\partial(\pi)$ is the cyclic permutation such that for any $1 \leq i \leq n-1$,
\[
\partial(\pi)(i)=
\begin{cases}
\pi(i) &\text{ if } \pi(i) \neq n, \\
\pi(n) &\text{ if } \pi(i)=n.
\end{cases}
\]
In words, $\partial$ deletes the largest number from the cycle. For example, $\partial([1,5,3,4,2])=[1,3,4,2]$.
For any $n\geq2$, we define the map $d:\calS_n\rightarrow\calS_{n-1}$ as follows: for any $\sigma\in\calS_n$, $d(\sigma)$ is the permutation such that for any $1 \leq i \leq n-1$,
\[
d(\sigma)(i)=
\begin{cases}
\sigma(i) &\text{ if } \sigma(i)<\sigma(n), \\
\sigma(i)-1 &\text{ if } \sigma(i)>\sigma(n).
\end{cases}
\]
For example, using the one-line notation for permutations, we have $d(1 4 2 5 3)=1 3 2 4$.

Note that
\[
D:
\begin{array}{ccc}
 \calS_n & \longrightarrow & \calS_{n-1}\times\left\{1,\ldots,n\right\} \\
 && \\
 \sigma & \longmapsto & (d(\sigma),\sigma(n))
\end{array}
\]
is a bijection.

We will define the bijection $F$ by induction on the size of the cyclic permutation $\pi$. There is a unique cyclic permutation $\pi_0\in\calC_2$. We set $F(\pi_0)$ to be the unique permutation in $\calS_1$. Fix $n \geq 2$. Assume we have defined how $F$ acts on cyclic permutations in $\calC_n$. Fix now $\pi\in\calC_{n+1}$ and write $\tilde{\sigma}=F(\partial(\pi))$. Set
\begin{equation}
\beta:=c_{\zeta(\pi)}(n,{n+1})
\end{equation}
and write
\begin{equation}
\label{eq:defalpha}
\alpha=
\begin{cases}
n-\beta \text{ if } n \text{ is even}, \\
1+\beta \text{ if } n \text{ is odd}.
\end{cases}
\end{equation}
Then we set
\begin{equation}
F(\pi)=D^{-1}(\tilde{\sigma},\alpha).
\end{equation}
For example, $F([1,5,3,4,2])=4 3 1 2$.

Theorem~\ref{thm:cyclicshape} immediately follows from the next two lemmas.

\begin{lemma}
\label{lem:Phibijection}
The map $F$ induces for all $n\geq1$ a bijection from $\calC_{n+1}$ to $\calS_n$.
\end{lemma}

\begin{lemma}
\label{lem:Phiactiononclasses}
For any $n\geq1$ and $w\in\left\{+,-\right\}^n$, $F(\calC_w)\subset\calS_{i(w)}$.
\end{lemma}

Corollary~\ref{cor:Entringernumbers} about Entringer numbers follows from formula~\eqref{eq:defalpha} defining $\alpha$.

\subsection{Proof of Lemma~\ref{lem:Phibijection}.}
\label{subsec:lemma1}

The sets $\calC_{n+1}$ and $\calS_n$ both have cardinality $n!$. We will show by induction on $n\geq1$ that $F:\calC_{n+1}\rightarrow\calS_n$ is surjective. It is clear for $n=1$.

Pick $\sigma\in\calS_n$ with $n\geq 2$. By induction hypothesis, we can find $\tilde{\pi}\in\calC_n$ such that $F(\tilde{\pi})=d(\sigma)$. We set $\beta$ to be $n-\sigma(n)$ (resp. $\sigma(n)-1$) if $n$ is even (resp. odd). For any $1\leq i \leq n+1$, define
\[
\pi(i)=
\begin{cases}
n+1 &\text{ if } i=\tilde{\pi}^{\beta}(n), \\
\tilde{\pi}^{\beta+1}(n) &\text{ if } i=n+1, \\
\tilde{\pi}(i) &\text{ otherwise}.
\end{cases}
\]
It is not hard to check that $\pi\in\calC_{n+1}$. In words, $\pi$ is obtained from $\tilde{\pi}$ by adding the element $n+1$ to the cycle in such a way that $c_{\zeta(\pi)}(n,{n+1})=\beta$. Thus by construction of $F$ we have $F(\pi)=\sigma$, which concludes the proof of Lemma~\ref{lem:Phibijection}.

\subsection{Proof of Lemma~\ref{lem:Phiactiononclasses}.}
\label{subsec:lemma2}

We will show by induction on $n\geq1$ that for any $w\in\left\{+,-\right\}^n$ and $\pi\in\calC_w$, $F(\pi)\in\calS_{i(w)}$. Since $F([1,2,3])=(1 2)$ and $F([3,2,1])=(2 1)$, it is true for $n=1$.
Define the map $\delta:\left\{+,-\right\}^n\rightarrow\left\{+,-\right\}^{n-1}$ by $\delta(\epsilon_1\cdots \epsilon_n):=\epsilon_1\cdots \epsilon_{n-1}$.
Pick $n\geq2$, $w\in\left\{+,-\right\}^n$ and $\pi\in\calC_w$. Set $\sigma=F(\pi)$, $\tilde{\pi}=\partial(\pi)$ and $\tilde{\sigma}=F(\tilde{\pi})$. Then $\tilde{\pi}\in\calC_{\delta(w)}$ and by induction hypothesis $\tilde{\sigma}\in\calS_{i(\delta(w))}$. Observing that the maps $i$ and $\delta$ commute and that $\tilde{\sigma}=d(\sigma)$, we get that $d(\sigma)\in\calS_{\delta(i(w))}$. For any $1 \leq i \leq n-1$, we have $\sigma(i)<\sigma(i+1)$ if and only if $d(\sigma)(i)<d(\sigma)(i+1)$, so in order to conclude that $\sigma\in\calS_{i(w)}$, it suffices to show that $\sigma(n)<\sigma(n+1)$ (resp. $\sigma(n)>\sigma(n+1)$) if the last letter of $i(w)$ is $+$ (resp. $-$). Set
\begin{align}
\beta_1 &= c_{\zeta(\tilde{\pi})}(n,{n+1}) \\
\beta_2 &= c_{\zeta(\pi)}({n+1},{n+2}).
\end{align}
Observe that
\begin{equation}
\beta_1=\#\left\{1\leq i \leq n-1 | (n,i,{n+1})\in \zeta(\pi) \right\}.
\end{equation}
We now need to distinguish according to the parity of $n$ and the last letter $\epsilon_n$ of $w$.

Assume first $n$ is odd. Observe that $\sigma(n+1)=n+1-\beta_2$ and
\begin{equation}
\sigma(n)=
\begin{cases}
\beta_1 +1 &\text{ if } \beta_1 +1<\sigma(n+1), \\
\beta_1 +2 &\text{ if } \beta_1 +1 \geq\sigma(n+1).
\end{cases}
\end{equation}
In particular, $\sigma(n)<\sigma(n+1)$ if and only if $\beta_1+\beta_2 \leq n-1$.
If $\epsilon_n=+$ (resp. $\epsilon_n=-$), then $(n,{n+1},{n+2})\in \zeta(\pi)$ (resp. $({n+2},{n+1},n)\in \zeta(\pi)$) thus $\beta_1+\beta_2\leq n-1$ (resp. $\beta_1+\beta_2\geq n$) hence $\sigma(n)<\sigma(n+1)$ (resp. $\sigma(n)>\sigma(n+1)$), which concludes the proof in this case since the last letter of $i(w)$ is $\epsilon_n$.

Assume now that $n$ is even. This time $\sigma(n+1)=1+\beta_2$ and
\begin{equation}
\sigma(n)=
\begin{cases}
n-\beta_1 &\text{ if } n-\beta_1<\sigma(n+1), \\
n-\beta_1+1 &\text{ if } n-\beta_1 \geq \sigma(n+1).
\end{cases}
\end{equation}
In particular, $\sigma(n)<\sigma(n+1)$ if and only if $\beta_1+\beta_2 \geq n$. Just as in the case of $n$ odd, we find that if $\epsilon_n=+$ (resp. $\epsilon_n=-$), then $\sigma(n)>\sigma(n+1)$ (resp. $\sigma(n)<\sigma(n+1)$), which concludes the proof in this case since the last letter of $i(w)$ is $-\epsilon_n$.

\section{Proof of the recursion formulas for \texorpdfstring{$\calQ_w^{\eta}$}{Q,w,eta} and \texorpdfstring{$\calR_w^{(\alpha)}$}{R,w,alpha}}
\label{sec:proof2}

We first show how to derive linear recursion relations for the triangles and tetrahedra of numbers, then how to translate them to the language of generating functions.

\subsection{Recursion relations for triangles and tetrahedra of numbers}

The proof of the linear recursion formulas for $\calQ_w^{\eta}$  (resp. for $\calR_w^{(\alpha)}$) goes along the following lines: take a total cyclic order $Z$ on $[n+1]$ where the elements $n$, ${n+1}$ and $1$ (resp. $n,{n+1},1$ and $2$) form a prescribed chain in $Z$, obtain a total cyclic order $\tilde{Z}$ on $[n]$ by deleting the element ${n+1}$ and look at the possible chains formed in $\tilde{Z}$ by ${n-1}$, $n$ and $1$ (resp. ${n-1},n,1$ and $2$).

Define the map
\begin{equation}
\bar{\partial}:
\begin{array}{ccc}
\bigsqcup_{n\geq4} \calZ_n & \longrightarrow & \bigsqcup_{n\geq3} \calZ_n \\
&&\\
 Z & \longmapsto & \zeta \circ \partial \circ \zeta^{-1}(Z)
\end{array}.
\end{equation}
In words, for any $n\geq3$ and $Z\in\calZ_{n+1}$, $\bar{\partial}(Z)$ is the total order on $[n]$ obtained by deleting from $Z$ all the triples involving ${n+1}$.

For any $n\geq3$, $w\in\left\{+,-\right\}^n$ and $i,j,k\geq0$ with $i+j+k=n-3$, set
\begin{align}
\calQ_{w,i,j,k}^+ &:=\left\{ Z\in\calQ_w^+ | \tilde{c}_Z({n-1},n,1)=(i,j,k) \right\} \\
\calQ_{w,i,j,k}^- &:= \left\{ Z\in\calQ_w^- | \tilde{c}_Z(n,{n-1},1)=(i,j,k) \right\}.
\end{align}

\begin{lemma}
\label{lem:removing}
For any $n\geq1$, $w\in\left\{+,-\right\}^n$ and $i,j,k\geq0$ such that $i+j+k=n$, the restriction of the map $\bar{\partial}$ to $\calQ_{w+,i,j,k}^+$ is a bijection from $\calQ_{w+,i,j,k}^+$ to the set
\[
\tilde{\calQ}_{w,i,j,k}:=\bigsqcup_{j'=0}^{j-1} \calQ_{w,i+j-1-j',j',k}^- \sqcup \bigsqcup_{k'=0}^{k-1} \calQ_{w,k-1-k',i+j,k'}^+.
\]
\end{lemma}

\begin{proof}
Fix $n\geq3$, $w\in\left\{+,-\right\}^{n-2}$ and $Z\in\calQ_{w+,i,j,k}^+$. We first show that $\bar{\partial} Z$ lies in $\tilde{\calQ}_{w,i,j,k}$. Since $\bar{\partial} Z$ lies in $\calZ_n$ and the relative order of any triple $i,{i+1},{i+2}$ with $1\leq i\leq n-2$ is prescribed by the word $w$, $\bar{\partial} Z$ must lie either in some $\calQ_{w,i',j',k'}^-$ or in some $\calQ_{w,i',j',k'}^+$. Assume first that $\bar{\partial} Z$ lies in some $\calQ_{w,i',j',k'}^-$ (see Figure~\ref{fig:finduction}).

\begin{figure}[htpb]
\centering
\includegraphics[height=2.5in]{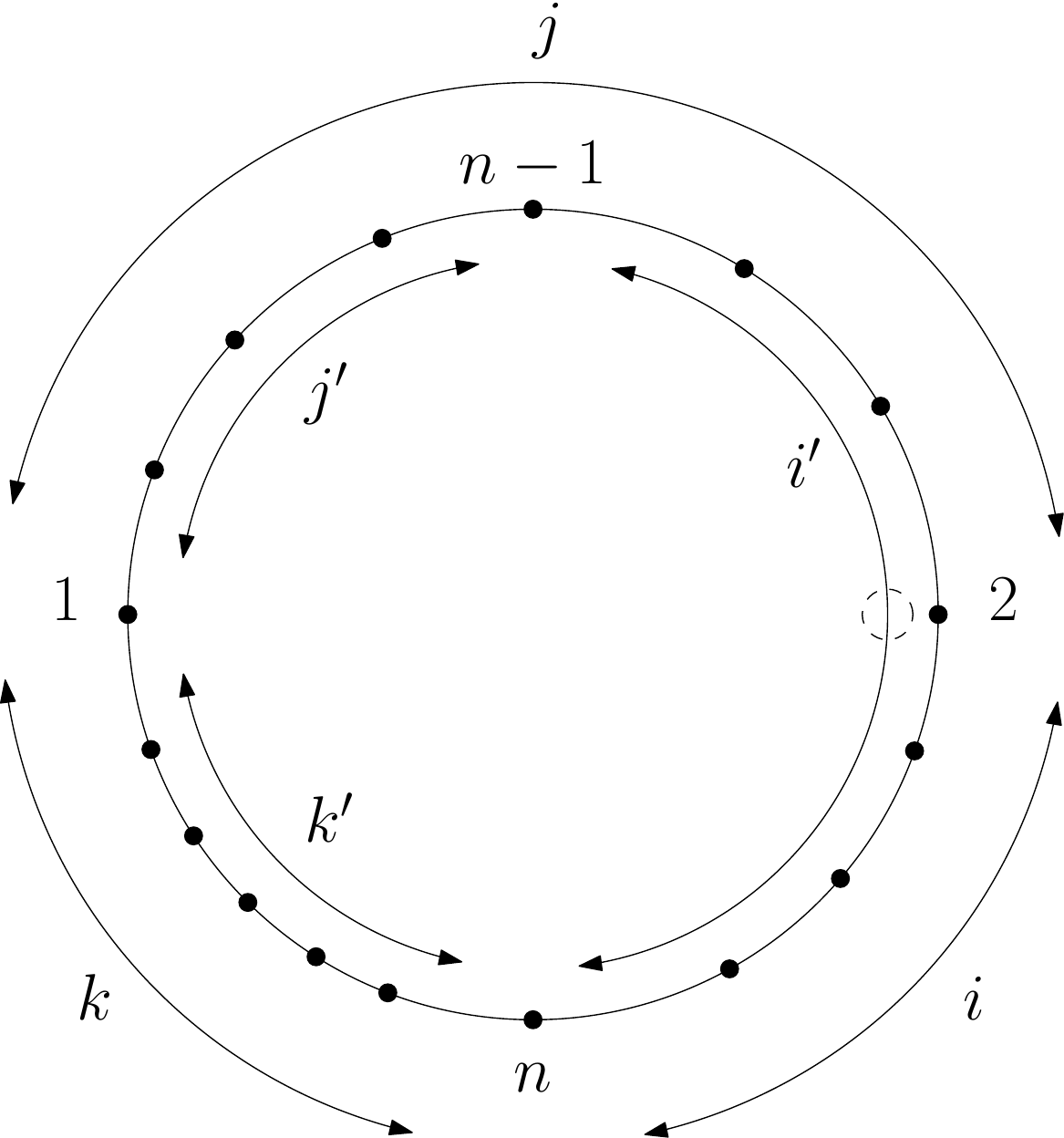}
\caption{A cyclic order $Z$ in some $\calQ_{w+,i,j,k}^+$ with $w\in\left\{+,-\right\}^{n-2}$ and such that $\bar{\partial}Z\in\calQ_{w,i',j',k'}^-$. The small dashed circle indicates that ${n+1}$ is not counted by $i'$.}
\label{fig:finduction}
\end{figure}

Then $({n-1},n,{n+1})\in Z$ and $({n-1},1,n)\in\bar{\partial} Z\subset Z$ thus by transitivity, it follows that $({n-1},1,n,{n+1})$ forms a chain in $Z$. From this we deduce the following:
\begin{gather*}
k'=c_{\bar{\partial}Z}(1,n)=c_Z(1,n)=k ; \\
j'=c_{\bar{\partial}Z}({n-1},1)=c_Z({n-1},1) < c_Z({n+1},1)=j ;  \\
i'+j'=c_{\bar{\partial}Z}(n,{n-1})+c_{\bar{\partial}Z}({n-1},1)=c_Z(n,{n-1})-1+c_Z({n-1},1) \\
=c_Z(n,{n+1})+c_Z({n+1},1)-1=i+j-1.
\end{gather*}
Thus
\[
\bar{\partial}Z\in\bigsqcup_{j'=0}^{j-1} \calQ_{w,i+j-1-j',j',k}^-.
\]
Similarly, in the case when $\bar{\partial} Z$ lies in some $\calQ_{w,i',j',k'}^+$, we obtain that
\[
\bar{\partial}Z\in\bigsqcup_{k'=0}^{k-1} \calQ_{w,k-1-k',i+j,k'}^+.
\]

For any $\tilde{Z}\in \tilde{\calQ}_{w,i,j,k}$, let $\iota(\tilde{Z})$ be the total order on $[n+1]$ obtained from $\tilde{Z}$ by adding ${n+1}$ on the circle in such a way that
\[c_{\iota(\tilde{Z})}(n,{n+1})=i.
\]
More precisely, writing $\tilde{\pi}=\zeta^{-1}(\tilde{Z})$, set $\iota(\tilde{Z}):=\zeta(\pi)$, where $\pi$ is defined by
\[
\pi(\ell)=
\begin{cases}
n+1 &\text{ if } \ell=\tilde{\pi}^{i}(n), \\
\tilde{\pi}^{i+1}(n) &\text{ if } \ell=n+1, \\
\tilde{\pi}(\ell) &\text{ otherwise}.
\end{cases}
\]
It is not hard to see that $\iota$ is a left- and right-inverse to the restriction of the map $\bar{\partial}$ to $\calQ_{w+,i,j,k}^+$, which is thus a bijection.
\end{proof}

The recurrence relation~\eqref{eq:firstlinearrecurrence} from Theorem~\ref{thm:extendedcoefficients} immediately follows from Lemma~\ref{lem:removing}. The other statements of this Theorem are proved using statements analogous to Lemma~\ref{lem:removing}. For the recurrence relation~\eqref{eq:secondlinearrecurrence}, one should observe that the image under $\bar{\partial}$ of $\calQ_{w+,i,j,k}^-$ must lie in some $\calQ_{w,i',j',k'}^+$, using transitivity: if $Z\in\calQ_{w+,i,j,k}^-$, then $(n,1,{n+1})\in Z$ and $(n,{n+1},{n-1})\in Z$ thus $(n,1,{n-1}) \in Z$ and $(n,1,{n-1}) \in \bar{\partial}Z$. Finally, one obtains similarly linear recursion formulas for the $g_{w,i,j,k,\ell}^{(\alpha)}$.

\subsection{Recursion relations for generating functions}

Before we translate the recursion relations for triangles and tetrahedra of numbers into recursion relations for generating functions, we need to introduce some notation for certain subsets of indices. Fix $m\geq2$ and $1\leq a,b,c\leq m$ such that $b\neq c$. For any $\underline{i}=(i_1,\ldots,i_m)$ an $m$-tuple of nonnegative integers, we define the set $I_{a,b,c}(\underline{i})$ as follows.
 \begin{enumerate}
  \item If $a=b$, then $I_{a,a,c}(\underline{i})$ is the set of all $m$-tuples of nonnegative integers $\underline{i'}=(i'_1,\ldots,i'_m)$ verifying the following conditions:
  \begin{itemize}
   \item $0 \leq i'_a \leq i_a-1$ ;
   \item $i'_c=i_c+i_a-1-i'_a$ ;
   \item $i'_\ell=i_\ell$ if $\ell\notin\left\{a,c\right\}$.
  \end{itemize}
 \item If $a \neq b$, then $I_{a,b,c}(\underline{i})$ is the set of all $m$-tuples of nonnegative integers $\underline{i'}=(i'_1,\ldots,i'_m)$ verifying the following conditions:
  \begin{itemize}
   \item $0 \leq i'_a \leq i_a-1$ ;
   \item $i'_b=i_a-1-i'_a$ ;
   \item $i'_c=i_b+i_c$ ;
   \item $i'_\ell=i_\ell$ if $\ell\notin\left\{a,b,c\right\}$.
  \end{itemize}
 \end{enumerate}
For any $n\geq0$, denote by $J_n$ the set of all $m$-tuples of nonnegative integers $\underline{i}=(i_1,\ldots,i_m)$ such that $i_1 + \cdots + i_m=n$.

In order to translate the recursion formulas for triangular or tetrahedral arrays of numbers into recursion formulas for generating functions, we apply the following lemma:

\begin{lemma}
\label{lem:generatingfunction}
Fix $n\geq0$ and let $(\lambda_{\underline{i'}})_{\underline{i'}\in J_n}$ and $(\mu_{\underline{i}})_{\underline{i}\in J_{n+1}}$ be two collections of numbers indexed by $m$-tuples of nonnegative integers.
Assume that for any $\underline{i}\in J_{n+1}$, we have
\begin{equation}
\mu_{\underline{i}}=\sum_{\underline{i'}\in I_{a,b,c}(\underline{i})} \lambda_{\underline{i'}}
\end{equation}
Then we have
\begin{equation}
 \Phi_{a,b,c}\left(\sum_{\underline{i'}\in J_n}\lambda_{\underline{i'}}\prod_{\ell=1}^m X_{\ell}^{i'_\ell}\right)=\sum_{\underline{i}\in J_{n+1}}\mu_{\underline{i}}\prod_{\ell=1}^m X_{\ell}^{i_\ell}.
\end{equation}
 \end{lemma}

The proof of Lemma~\ref{lem:generatingfunction} is a straightforward computation.

\section{A conjecture on asymptotic densities}
\label{sec:conjecture}

Observe that for any $n\geq1$ and $w\in\left\{+,-\right\}^n$, we have
\[
\calP_w=\bigsqcup_{\alpha=1}^6 \calR_w^{(\alpha)}.
\]
One may investigate the density of each $\calR_w^{(\alpha)}$ inside $\calP_w$. In particular, in the case when $w=+^n$ (in which case $\#\calP_w$ is an Euler number), denote by
\[
p_n^{(\alpha)}:=\frac{\#\calR_{+^n}^{(\alpha)}}{\#\calP_{+^n}}
\]
the density of each $\calR_{+^n}^{(\alpha)}$ inside $\calP_{+^n}$ for any $1 \leq \alpha \leq 6$. We conjecture the following regarding the asymptotics of $p_n^{(\alpha)}$:

\begin{conjecture}
\label{conj:asymptoticdensities}
For any $1 \leq \alpha \leq 6$, $p_{\infty}^{(\alpha)}:=\lim_{n\rightarrow\infty} p_n^{(\alpha)}$ exists and is given by:
\begin{gather}
p_{\infty}^{(1)}=\frac{1}{\pi} \\
p_{\infty}^{(2)}=p_{\infty}^{(5)}=\frac{1}{2}-\frac{1}{\pi} \\
p_{\infty}^{(3)}=p_{\infty}^{(4)}=\frac{2}{\pi}-\frac{1}{2} \\
p_{\infty}^{(6)}=1-\frac{3}{\pi}
\end{gather}
\end{conjecture}
The conjecture is supported by numerical simulations, whereby we computed each $p_n^{(\alpha)}$ for $n \leq 50$ and $1 \leq \alpha \leq 6$ and we observed the convergence to the predicted values, with a precision of $10^{-8}$ when $n=50$. If true, Conjecture~\ref{conj:asymptoticdensities} would imply that as $n$ goes to infinity, the asymptotic density of $\calQ_{+^n}^+$ (resp. $\calR_{+^n}^{+,+}$) inside $\calP_{+^n}$ equals $\frac{2}{\pi}$ (resp. $\frac{1}{2}$).

\paragraph*{Acknowledgements}

We thank Arvind Ayyer, Dan Betea, Wenjie Fang, Mat-thieu Josuat-Verg\`es and Bastien Mallein for fruitful discussions. We also acknowledge the support and hospitality of the Institut Henri Poincar\'e, where this work was initiated during the program on ``Combinatorics and interactions'', as well as the support of the Fondation Simone et Cino Del Duca. Finally we thank the anonymous referee for providing advice to improve the exposition.

\label{Bibliography}
\bibliographystyle{plain}
\bibliography{bibliographie}

\Addresses
\end{document}